\documentclass[a4paper,12pt]{amsart}
\usepackage{amssymb}
\usepackage{amsmath}
\usepackage{graphics}
\usepackage{soul}
\usepackage{xcolor}

\usepackage[dvips]{graphicx}
\newtheorem{thm}{Theorem}
\newtheorem{ej}{Example}

\newtheorem{lemma}[thm]{Lemma}
\newtheorem{prop}[thm]{Proposition}

\newtheorem{rem}{Remark}

\def\C{{\mathbb C}}

\def\D{{\mathbb D}}

\newcommand{\RR}{\text{Re}}

\def\be{\begin{equation}}
\def\ee{\end{equation}}

\makeatletter

\begin{document}

\title[Schwarzian derivative for convex mappings of order alpha]{Schwarzian derivative for convex mappings of order alpha}

\author{P. Carrasco\and R. Hernández }
\thanks{The
authors were partially supported by Fondecyt Grants \# 1190756.
\endgraf  {\sl Key words:} {Convex Mappings, Schwarzian derivative, Schwarz Lemma.}
\endgraf {\sl 2020 AMS Subject Classification}. Primary:30C45, 30C80; \,
Secondary:  30A10.}

\address{Facultad de Ingenier\'ia y Ciencias\\
Universidad Adolfo Ib\'a\~nez\\
Av. Padre Hurtado 750, Vi\~na del Mar, Chile.}
\email{pablo.carrasco@edu.uai.cl}

\address{Facultad de Ingenier\'ia y Ciencias\\
Universidad Adolfo Ib\'a\~nez\\
Av. Padre Hurtado 750, Vi\~na del Mar, Chile.}
\email{rodrigo.hernandez@uai.cl}


\begin{abstract} The main purpose of this paper is to obtain sharp bounds of the norm of Schwarzian derivative for convex mappings of order $alpha$ in terms of the value of $f''(0)$, in particular, when this quantity is equal to zero. In addition, we obtain sharp bounds for distortion and growth for this mappings and we generalized the results obtained by Suita \cite{SUITA} and Yamashita \cite{Y} for this particular case.
\end{abstract}

\maketitle

\section{Introduction}

Let $\mathcal{C}$ be the class of conformal functions such that $f(\D)$ is a convex region in the complex plane, normalized by the conditions $f(0)=0$ and $f'(0)=1$, where $\D$ is the unit disk in $\C$. A very important characterization of these functions is that $f\in\mathcal{C}$ if and only if 
\begin{equation*}
    \RR\left\{1+z\,\dfrac{f''}{f'}(z)\right\}>0,\quad z\in\mathbb{D}.
\end{equation*}
Nehari in \cite{NEHARI76} shows that if $f\in\mathcal{C}$ then
\begin{equation}\label{nehari}
    (1-\lvert z\rvert ^2)^2\,\lvert Sf(z)\rvert\leq 2,\quad z\in\D,
\end{equation}
where the differential operator $Sf$ is known as the Schwarzian derivative of $f$ and it is defined as:
\begin{equation*}
    Sf(z)=\left(\dfrac{f''}{f'}(z)\right)'-\dfrac{1}{2}\left(\dfrac{f''}{f'}(z)\right)^2,
\end{equation*}
considering that $f$ is locally univalent in $\D$. Inequality (\ref{nehari}) was obtained by Nehari applying geometric arguments via the Schwarz-Christoffel formula for convex polygons. In the same time, Trimble in \cite{TRIMBLE} showed that if $f(z)=z+a_2z^2+a_3z^3+\cdots$ maps $\D$ into a convex domain, then $|a_3-a_2^2|\leq(1-|a_2|^2)/3$. Since $\mathcal C$ is a linear invariant family (see Remark \ref{remark_1}) this inequality implies also (\ref{nehari}). Later, Chuaqui et al. in \cite{CDO} proved the same result by applying the Schwarz-Pick lemma and the the fact that the expression  $1+z(f''/f')(z)$ is subordinate to the half plane mapping $\ell(z)=(1+z)/(1-z)$, which is
\begin{equation}\label{convex_subo}
    1+z\,\dfrac{f''}{f'}(z)=\ell(w(z))=\dfrac{1+w(z)}{1-w(z)},
\end{equation}
for some function $w:\D\to\D$ holomorphic and such that $w(0)=0$. The expression defined in (\ref{convex_subo}) allowed the authors to obtain other characterizations for convex functions:
\begin{equation}\label{charac1_convex}
    f\in\mathcal{C}\text{ if and only if, } \RR\left\{1+z\,\dfrac{f''}{f'}(z)\right\}\geq \dfrac{1}{4}(1-\lvert z\rvert ^2)\left\lvert\dfrac{f''}{f'}(z)\right\rvert^2,
\end{equation}
and
\begin{equation}\label{charac2_convex}
    f\in\mathcal{C}\text{ if and only if, } \left\lvert(1-\lvert z\rvert^2)\dfrac{f''}{f'}(z)-2\bar{z}\right\rvert\leq 2,
\end{equation}
for all $z\in\D$. 

One of the many extensions of the class of convex functions was given by Robertson in \cite{ROBERTSON}: $f$ is a convex function of order $\alpha$, with $0\leq\alpha<1$, if
\begin{equation*}\label{convex_alpha}
    \RR\left\{1+z\,\dfrac{f''}{f'}(z)\right\}>\alpha.
\end{equation*}
We denote the class of convex function of order $\alpha$ as $\mathcal{C}_\alpha$. Observe that $\mathcal{C}_0=\mathcal{C}$ and $\mathcal{C}_\alpha\subset \mathcal{C}$ for all $0\leq\alpha<1$. Robertson \cite{ROBERTSON} observed that functions $f$ in $\mathcal C_\alpha$, have the following geometric property: the ratio of the angle between adjacent tangents of the unit circle over the
angle between the corresponding tangents in the image of $f$ is less than $1/\alpha$. Hence the closer $\alpha$ is to 1 the ``rounder'' is the image. For instance, segments in the boundary of $f(\D)$ are prohibited as soon as $\alpha > 0$. An analogous bound associated to (\ref{nehari}) was obtained by Suita in \cite{SUITA} for this class. We present this result as Theorem A in section \ref{norm of Sf}.The author obtains his result characterizing the expression $1+z(f''/f')(z)$ in terms of their Herglotz representation associated to the Poisson Kernel (see \cite[p. 15]{DUREN_UF}). In \cite[p. 54]{GRAHAM_KOHR} there is a review of the most important properties of the class $\mathcal{C}_\alpha$.

In this work, we apply some of the ideas presented in \cite{CDO}, in particular the Schwarz Lemma and its generalizations, to obtain analogous inequalities associated to (\ref{charac1_convex}) and (\ref{charac2_convex}) for the convex of order $\alpha$ case, and a sharp bound for the norm of the Schwarzian derivative when $f\in \mathcal C_\alpha$ with the additional condition that $f''(0)=0$. In the same path, Kanas and Sugawa in \cite{KS} proved similar results, but for the case of strongly convex mappings, which can not be improved with our additional assumption since the extremal mapping satisfies that $f''(0)=0$. Yamashita in \cite{Y} showed that the norm of the pre-Schwarzian derivative for convex mappings with order alpha is less than or to $4(1-\alpha)$, which is improved in the Proposition \ref{prop_3} assuming that $f''(0)=0$. In fact, we shall show that $(1-\lvert z\rvert^2)\,\lvert f''(z)/f'(z)\rvert\leq 2(1-\alpha)$. Also, for restricted values of $\lvert f''(0)\rvert\neq 0$, we obtain a better bound than the one presented by Suita for $\lVert Sf\rVert$ in the general case.

\section{Bounds on convex mappings with order $\alpha$.}

The M\"obius transformation:
\begin{equation*}
    \ell_\alpha(z)=\dfrac{1+(1-2\alpha)z}{1-z},
\end{equation*}
sends $\D$ to the half plane $H_\alpha=\{z\in\C:\RR(z)>\alpha\}$. If $f\in\mathcal{C}_\alpha$ then there exists $w:\D\rightarrow\D$ such that:
\begin{equation*}
    1+z\,\dfrac{f''}{f'}(z)=\ell_\alpha(w(z))=\dfrac{1+(1-2\alpha)w(z)}{1-w(z)},
\end{equation*}
 with $w(0)=0$, analogously as in (\ref{convex_subo}). We can assume that $w(z)=z\varphi(z)$ for some holomorphic mapping $\varphi$ that satisfies $\varphi(\D)\subseteq\D$, from where:
\begin{equation}\label{preschwarz_alpha}
\dfrac{f''}{f'}(z)=2(1-\alpha)\dfrac{\varphi(z)}{1-z\varphi(z)},
\end{equation}

The next theorem generalizes the characterizations (\ref{charac1_convex}) and (\ref{charac2_convex}) to convex functions of order $\alpha$.

\begin{thm}\label{teo_1}
If $0\leq\alpha<1$, then the following are equivalent:

\begin{enumerate}
\item[i.] $f\in\mathcal{C}_\alpha$.
\item[ii.] $\RR\left\{1+z\dfrac{f''}{f'}(z)\right\}\geq \alpha+\dfrac{1}{4(1-\alpha)}(1-\lvert z\rvert^2)\left\lvert\dfrac{f''}{f'}(z)\right\rvert^2$.
\item[iii.] $\left\lvert(1-\lvert z\rvert^2)\dfrac{f''}{f'}(z)-2(1-\alpha)\bar{z}\right\rvert\leq 2(1-\alpha)$.
\end{enumerate}
\end{thm}
\begin{proof}
We prove that $i.$ is equivalent to $ii.$ and $ii.$ is equivalent to $iii.$. From (\ref{preschwarz_alpha}) we obtain that:
\begin{equation}\label{phi_preschwarz}
    \varphi(z)=\dfrac{(f''/f')(z)}{2(1-\alpha)+z(f''/f')(z)}.
\end{equation}
From the fact that $\lvert\varphi(z)\rvert^2\leq 1$ we have:
\begin{equation*}
    4(1-\alpha)^2+2\,\RR\left\{2(1-\alpha)z\,\dfrac{f''}{f'}(z)\right\}+\lvert z\rvert^2\left\lvert\dfrac{f''}{f'}(z)\right\rvert^2\geq \left\lvert\dfrac{f''}{f'}(z)\right\rvert^2.
\end{equation*}
Factorizing we obtain,
\begin{equation}\label{equa_1}
    4(1-\alpha)\left[1-\alpha+\RR\left\{z\,\dfrac{f''}{f'}(z)\right\}\right]\geq (1-\lvert z\rvert^2)\left\lvert\dfrac{f''}{f'}(z)\right\rvert^2.
\end{equation}
Since $\alpha\neq 1$ we conclude that
\begin{equation*}
    \RR\left\{1+z\,\dfrac{f''}{f'}(z)\right\}\geq \alpha+\dfrac{1}{4(1-\alpha)}(1-\lvert z\rvert^2)\left\lvert\dfrac{f''}{f'}(z)\right\rvert^2.
\end{equation*}
If we multiply both sides of inequality (\ref{equa_1}) by $(1-\lvert z\rvert^2)$, we obtain:
\begin{equation*}
    (1-\lvert z\rvert^2)^2\left\lvert\dfrac{f''}{f'}(z)\right\rvert^2\leq 4(1-\alpha)^2(1-\lvert z\rvert^2)+4(1-\alpha)(1-\lvert z\rvert^2)\RR\left\{\dfrac{f''}{f'}(z)\right\}.
\end{equation*}
This is equivalent to:
\begin{equation*}
    \left((1-\lvert z\rvert^2)\left\lvert\dfrac{f''}{f'}(z)\right\rvert\right)^2-2\,\RR\left\{2(1-\alpha)z(1-\lvert z\rvert^2)\dfrac{f''}{f'}(z)\right\}+(2(1-\alpha)\lvert z\rvert)^2\leq 4(1-\alpha)^2,
\end{equation*}
from which we have that:
\begin{equation*}
    \left\lvert(1-\lvert z\rvert^2)\dfrac{f''}{f'}(z)-2(1-\alpha)\bar{z}\right\rvert\leq 2(1-\alpha).
\end{equation*}
\end{proof}

It is a well-known fact that if $f\in \mathcal C$, then for any $a\in\D$ the corresponding \textit{Koebe Transform} $f_a$ given by   \begin{equation}\label{koebe_trans}
    f_a(z)=\dfrac{f\left(\dfrac{z+a}{1+\bar az} \right)-f(a)}{(1-\lvert a\rvert^2)f'(a)},
    \end{equation}
belongs to $\mathcal C$. However, if $f\in\mathcal{C}_\alpha$ then it is not necessarily that $f_a\in \mathcal C_\alpha$. For this case we have the following proposition:

\begin{prop} If $f\in \mathcal C_\alpha$ then for any $a\in\D$, $f_a\in\mathcal C_\beta$ with $\beta=\alpha\left(\dfrac{1-\lvert a\rvert}{1+\lvert a\rvert}\right)$.
\end{prop}
\begin{proof} By equation (\ref{koebe_trans}) we have \begin{equation*}
    \dfrac{f''_a}{f'_a}(z)=\dfrac{f''}{f'}\left(\dfrac{z+a}{1+\bar az} \right)\,\dfrac{1-\lvert a\rvert^2}{(1+\bar{a}z)^2}-\dfrac{2\bar{a}}{1+\bar{a}z}.
\end{equation*}
If we assume that $\lvert z\rvert=1$ (because of the maximum modulus principle) we have that:
\begin{align*}
    1+z\,\dfrac{f''_a}{f'_a}(z) & = 1+\left(\dfrac{z+a}{1+\bar az} \right)\,\dfrac{f''}{f'}\left(\dfrac{z+a}{1+\bar az} \right)\,\dfrac{z(1-\lvert a\rvert^2)(1+\bar{a}z)}{(1+\bar{a}z)^2(z+a)}-\dfrac{2\bar{a}z}{1+\bar{a}z}\\
    & = 1+\left(\dfrac{z+a}{1+\bar az} \right)\,\dfrac{f''}{f'}\left(\dfrac{z+a}{1+\bar az} \right)\,\dfrac{1-\lvert a\rvert^2}{\lvert 1+\bar{a}z\rvert^2}-\dfrac{2\bar{a}z}{1+\bar{a}z}.
\end{align*}
Let $\zeta=\left(\dfrac{z+a}{1+\bar az} \right)\in\D$, from which $ \RR\left\{\zeta\,\dfrac{f''}{f'}(\zeta)\right\}>\alpha-1$.
Then, for $\lvert z\rvert=1$ we have that
\begin{align*}
    \RR\left\{1+z\,\dfrac{f''_a}{f'_a}(z)\right\} & = 1+\left(\dfrac{1-\lvert a\rvert^2}{\lvert 1+\bar{a}z\rvert^2}\right)\RR\left\{\zeta\,\dfrac{f''}{f'}(\zeta)\right\}-2\,\RR\left\{\dfrac{\bar{a}z}{1+\bar{a}z}\right\}\\
    & \geq 1+\left(\dfrac{1-\lvert a\rvert^2}{\lvert 1+\bar{a}z\rvert^2}\right)(\alpha-1)-2\,\RR\left\{\dfrac{\bar{a}z}{1+\bar{a}z}\right\}\\
    & = 1+\alpha\left(\dfrac{1-\lvert a\rvert ^2}{\lvert 1+\bar{a}z\rvert^2}\right)-\left(\dfrac{1-\lvert a\rvert^2+2\,\RR\{\bar{a}z(1+a\bar{z}\}}{\lvert 1+\bar{a}z\rvert^2}\right)\\
    & = 1+\alpha\left(\dfrac{1-\lvert a\rvert ^2}{\lvert 1+\bar{a}z\rvert^2}\right)-\left(\dfrac{1+2\,\RR\{\bar{a}z\}+\lvert a\rvert^2}{\lvert 1+\bar{a}z\rvert^2}\right)\\
    & = \alpha\left(\dfrac{1-\lvert a\rvert^2}{\lvert 1+\bar{a}z\rvert^2}\right)\geq  \alpha\left(\dfrac{1-\lvert a\rvert^2}{(1+\lvert a\rvert)^2}\right)=\alpha\left(\dfrac{1-\lvert a\rvert}{1+\lvert a\rvert}\right).
\end{align*}
From the above we can conclude that if $f\in\mathcal{C}_\alpha$ then $f_a\in\mathcal{C}_\beta$.
\end{proof}

Now, with the additional condition that $f''(0)=0$, equation (\ref{preschwarz_alpha}) shows that $\varphi$ satisfies that $\varphi(z)=z\psi(z)$, for some holomorphic mapping $\psi$ with $\lvert\psi\rvert<1$. We say that $f\in \mathcal C_\alpha^0$ if $f\in\mathcal C_\alpha$ and $f''(0)=0$. Thus, a straightforward calculation shows that:

\begin{prop}\label{prop_3} If $f\in \mathcal{C}_\alpha^0$ then \begin{equation*}
\left(1-\lvert z\rvert^2\right)\left\lvert\dfrac{f''}{f'}(z)\right\rvert\leq 2\lvert z\rvert(1-\alpha),\ \forall z\in\D.
\end{equation*}
\end{prop}

\begin{proof} Since $\varphi(z)=z\psi(z)$, with $\lvert\psi(z)\rvert<1$, then in equation (\ref{preschwarz_alpha}) we have:
\begin{equation*}
\left\lvert\dfrac{f''}{f'}(z)\right\rvert\leq 2(1-\alpha)\dfrac{\lvert z\rvert\,\lvert \psi(z)\rvert}{1-\lvert z\rvert^2\lvert\psi(z)\rvert}\leq 2(1-\alpha)\dfrac{\lvert z\rvert}{1-\lvert z\rvert^2}.
\end{equation*}
\end{proof}

Robertson in \cite{ROBERTSON} shows that if $f\in\mathcal C_\alpha$ then $$\dfrac{1}{(1+\lvert z\rvert)^{2(1-\alpha)}}\leq \lvert f'(z)\rvert\leq \dfrac{1}{(1-\lvert z\rvert)^{2(1-\alpha)}},$$ and, if $\alpha\neq 1/2$ 
$$\dfrac{(1+\lvert z\rvert)^{2\alpha-1}-1}{2\alpha-1}\leq \lvert f(z)\rvert\leq \dfrac{1-(1-\lvert z\rvert)^{2\alpha-1}}{2\alpha-1},$$ for $\alpha=1/2$, $$\log(1+\lvert z\rvert)\leq \lvert f(z) \rvert\leq -\log(1-\lvert z\rvert)\,.$$ With the additional condition $f''(0)=0$, we improve these results which are summarized in the following theorem.

\begin{thm}\label{teo_4} If $f\in\mathcal C^0_\alpha$ then for all $z\in\D$, we have that $$\dfrac{1}{(1+\lvert z\rvert^2)^{1-\alpha}}\leq \lvert f'(z)\rvert\leq \dfrac{1}{(1-\lvert z\rvert^2)^{1-\alpha}},$$ and $$\int_0^{\lvert z\rvert}\dfrac{1}{(1+\zeta^2)^{1-\alpha}}\,d\zeta\leq \lvert f(z)\rvert\leq \int_0^{\lvert z\rvert}\dfrac{1}{(1-\zeta^2)^{1-\alpha}}\,d\zeta.$$ \end{thm} We will see in Example \ref{ej_1}, that this bounds are sharp within $\mathcal C_\alpha^0$. In particular, when $\alpha=0$, which corresponds to the subclass $\mathcal C$ with $f''(0)=0$ denoted by $\mathcal C^0$, we have that $$\tan^{-1}(\lvert z\rvert)\leq \lvert f(z)\rvert\leq \frac12\log\left(\dfrac{1+\lvert z\rvert}{1-\lvert z\rvert}\right).$$ Clearly, these bounds are better than the well-known bounds for the entire convex class $\mathcal C$ (see \cite[p. 43]{GRAHAM_KOHR}). We will see in Theorem 5 that $\lVert Sf\rVert\leq 2(1-\alpha^2)$ for any $f\in \mathcal C_\alpha^0$, therefore, a comparison is in order between these bounds and those obtained by Chuaqui and Osgood in \cite{CO}. These authors proved certain sharp bounds for distortion and growth for mappings with the condition that $(1-\lvert z\rvert^2)^2\,\lvert Sf\rvert\leq 2t$ when $t\in[0,1]$. They showed that if $f$ is holomorphic in $\D$ such that $f(0)=0$, $f'(0)=1$ and $f''(0)=0$, then:
\begin{align*}
    A(\lvert z\rvert,-t) & \leq \lvert f(z)\rvert \leq A(\lvert z\rvert,t), \\
    A'(\lvert z\rvert,-t) & \leq \lvert f'(z)\rvert \leq A'(\lvert z\rvert,t),
\end{align*} where 

\begin{equation*}
    A(z,t)=\dfrac{1}{\sqrt{1-t}}\dfrac{(1+z)^{\sqrt{1-t}}-(1-z)^{\sqrt{1-t}}}{(1+z)^{\sqrt{1-t}}+(1-z)^{\sqrt{1-t}}}.
\end{equation*}

Our case ($\alpha=0$) corresponds to their case when $t=1$, where the upper bound for growth coincides with our upper bound because the mapping for which equality holds belongs to both classes. However, for the lower bound the situation is completely different since their lower bound is less than or equal to our lower bound. Moreover, when $\alpha=1/2$, the growth inequality in Theorem \ref{teo_4} is $\sinh^{-1}(\lvert z\rvert)\leq \lvert f(z)\rvert\leq \sin^{-1}(\lvert z\rvert),$ but the corresponding parameter is  $t=3/4$, hence our result is more precise than the one obtained by Chuaqui and Osgood in \cite{CO}, which is not surprising since in their case the family is larger than $\mathcal C_\alpha^0$. 

\begin{proof}[Proof of Theorem 4]

From equation (\ref{phi_preschwarz}) and considering that $\varphi(0)=0$ (because $f''(0)=0$), then by the Schwarz Lemma we have that
\begin{equation*}
    \left\lvert \dfrac{(f''/f')(z)}{2(1-\alpha)+z(f''/f')(z)}\right\rvert^2\leq \lvert z\rvert^2,
\end{equation*}
which is equivalent to
\begin{equation*}
    (1-\lvert z\rvert^4)\left\rvert\dfrac{f''}{f'}(z)\right\rvert^2\leq 4\lvert z\rvert^2(1-\alpha)^2+4\lvert z\rvert^2(1-\alpha)\RR\left\{z\dfrac{f''}{f'}(z)\right\}.
\end{equation*}
If we first multiply by $(1-\lvert z\rvert^4)$ and then we add $4(\lvert z\rvert^2(1-\alpha)\lvert \bar{z}\rvert)^2$ on both sides, we obtain
\begin{equation*}
\begin{split}
    (1&-\lvert z\rvert^4)^2\left\lvert\dfrac{f''}{f'}(z)\right\rvert^2 - 4\lvert z\rvert^2(1-\lvert z\rvert^4)(1-\alpha)\RR\left\{z\dfrac{f''}{f'}(z)\right\}+
     4(\lvert z\rvert^2(1-\alpha)\lvert \bar{z}\rvert)^2\\ \leq &{} 4(1-\lvert z\rvert^4)\lvert z\rvert^2(1-\alpha)^2+4\lvert z\rvert^4(1-\alpha)^2\lvert z\rvert^2=4\lvert z\rvert^2(1-\alpha)^2.
\end{split}
\end{equation*}
Then multiplying on both sides by $\lvert z\rvert $ we have
\begin{equation*}
    \left\lvert(1-\lvert z\rvert^4)z\dfrac{f''}{f'}(z)-2\lvert z\rvert^4(1-\alpha)\right\rvert\leq 2\lvert z\rvert^2 (1-\alpha).
\end{equation*}
This implies

\begin{equation*}
    \dfrac{-2\lvert z\rvert^2(1-\alpha)+2\lvert z\rvert^4(1-\alpha)}{1+\lvert z\rvert^4}\leq\RR\left\{z\dfrac{f''}{f'}(z)\right\}\leq\dfrac{2\lvert z\rvert^2(1-\alpha)+2\lvert z\rvert^4(1-\alpha)}{1-\lvert z\rvert^4}.
\end{equation*}
Which it is equivalent to

\begin{equation*}
    \dfrac{-2(1-\alpha)\lvert z\rvert^2}{1+\lvert z\rvert ^2}\leq\RR\left\{z\,\dfrac{f''}{f'}(z)\right\}\leq \dfrac{2(1-\alpha)\lvert z\rvert ^2}{1-\lvert z\rvert^2}.
\end{equation*}
Let $z=re^{i\theta}$, then we have
\begin{equation*}
    \dfrac{-2(1-\alpha)r}{1+r^2}\leq \dfrac{\partial}{\partial r} (\log\lvert f'(re^{i\theta})\rvert)\leq \dfrac{2(1-\alpha) r}{1-r^2}.
\end{equation*}
If we integrate respect to $r$, we obtain
\begin{equation*}
    \log(1+r^2)^{-(1-\alpha)}\leq \log\lvert f'(re^{i\theta})\rvert\leq \log(1-r^2)^{-(1-\alpha)}
\end{equation*}
Exponentiating we can conclude that
\begin{equation*}
    \dfrac{1}{(1+\lvert z\rvert^2)^{1-\alpha}}\leq \lvert f'(z)\rvert\leq \dfrac{1}{(1-\lvert z\rvert ^2)^{1-\alpha}}.
\end{equation*}
Now, for the growth part of the theorem, from the upper bound it follows that
\begin{equation*}
    \lvert f(re^{i\theta})\rvert = \left\lvert \int_0^r f'(t e^{i\theta})e^{i\theta}\,dt\right\rvert\leq \int_0^r\lvert f'(t e^{i\theta})\rvert\,dt \leq \int_0^r\dfrac{1}{(1-t^2)^{1-\alpha}}\,dt.\\
\end{equation*}
Which implies that
\begin{equation*}
    \lvert f(z)\rvert \leq \int_0^{\lvert z\rvert}\dfrac{1}{(1-\zeta^2)^{1-\alpha}}d\zeta,
\end{equation*} for all $z\in\D$. It is well known that if $f(z_0)$ is a point of minimum modulus on the image of the circle $\lvert z\rvert =r$ and $\gamma=f^{-1}(\Gamma)$, where $\Gamma$ is the line segment from $0$ to $f(z_0)$, then
\begin{equation*}
\lvert f(z)\rvert\geq \lvert f(z_0)\rvert =\int_\Gamma \lvert dw\rvert=\int_\gamma \lvert f'(\zeta)\rvert\,\lvert d\zeta\rvert\geq \int_0^r\dfrac{1}{(1+\lvert\zeta\rvert^2)^{1-\alpha}}\,d\lvert\zeta\rvert.
\end{equation*}
Thus, the proof is completed.
\end{proof}

\section{On the norm of the Schwarzian derivative when $f''(0)=0$.}\label{norm of Sf}

This section is devoted to finding the sharp bound of the norm of the Schwarzian Derivative in terms of the parameter $\alpha$ under the assumption that $f''(0)=0$. We define the $\mathcal{C}_\alpha^0$ class as those $f\in\mathcal{C}_\alpha$  such that $f''(0)=0$. Firstly, form equation (\ref{preschwarz_alpha}) it follows that
\begin{equation}\label{schwarz_alpha}
    Sf(z)=2(1-\alpha)\left(\dfrac{\varphi'(z)+\alpha\,\varphi^2(z)}{(1-z\varphi(z))^2}\right).
\end{equation} Suita in \cite{SUITA} proved the sharp bounds of $\lVert Sf\rVert$ for the hole class $\mathcal{C}_\alpha$. We present this results as the following Theorem:\\

\noindent\textbf{Theorem A.}\label{teo_A}\textit{
If $f\in\mathcal{C}_\alpha$ then  for all $z\in \D$ we have that if $0\leq\alpha\leq 1/2$:
\begin{equation*}
    (1-\lvert z\rvert^2)^2\,\lvert Sf(z)\rvert\leq 2,
\end{equation*}
and for $1/2<\alpha<1$:
\begin{equation*}
    (1-\lvert z\rvert^2)^2\,\lvert Sf(z)\rvert\leq 8\alpha(1-\alpha).
\end{equation*}}\\

The next theorem gives us a sharp bound for the expression $(1-\lvert z\rvert^2)^2\,\lvert Sf(z)\rvert$ when $f\in\mathcal{C}_\alpha^0$ by a direct application of the Schwarz Lemma.

\begin{thm}\label{teo_2} If $f\in \mathcal C_\alpha^0$ then 
\begin{equation*}
(1-\lvert z\rvert^2)^2\,\lvert Sf(z)\rvert\leq 2(1-\alpha^2),\ \forall z\in\D.
\end{equation*}
The inequality is sharp.
\end{thm}

\begin{proof}
From (\ref{schwarz_alpha}) we obtain, by applying the triangle inequality and the Schwarz-Pick lemma, that:
\begin{equation}\label{ineq_imp}
    (1-\lvert z\rvert^2)^2\,\lvert Sf(z)\rvert\leq 2(1-\alpha)\dfrac{(1-\lvert z\rvert^2)^2}{\lvert 1-z\varphi(z)\rvert^2}\left[\dfrac{1-\lvert\varphi(z)\rvert^2}{1-\lvert z\rvert^2}+\alpha\lvert\varphi(z)\rvert^2\right].
\end{equation}
We define the function $\Phi:\D\rightarrow\C$ such that:
\begin{equation*}
    \Phi(z)=\dfrac{\bar{z}-\varphi(z)}{1-z\varphi(z)}.
\end{equation*}
Since $\varphi(\D)\subset\D$, then $(1-\lvert z\rvert^2)(1-\lvert\varphi(z)\rvert^2)>0$, which implies that \begin{equation*}
    \lvert\bar{z}-\varphi(z)\rvert^2<\lvert1-z\varphi(z)\rvert^2,
\end{equation*}
so, we can conclude that $\lvert\Phi(z)\rvert^2<1$. Therefore,
\begin{equation*}
    1-\lvert\Phi(z)\rvert^2=\dfrac{(1-\lvert\varphi(z)\rvert^2)(1-\lvert z\rvert^2)}{\lvert1-z\varphi(z)\rvert^2},
\end{equation*}
and,
\begin{equation*}
    \dfrac{(1-\lvert z\rvert^2)^2}{\lvert1-z\varphi(z)\rvert^2}=\dfrac{(1-\lvert\Phi(z)\rvert^2)(1-\lvert z\rvert^2)}{(1-\lvert\varphi(z)\rvert^2)}.
\end{equation*}
If we replace the latter expression in (\ref{ineq_imp}) we have:
\begin{equation}\label{ineq_imp_2}
(1-\lvert z\rvert^2)^2\,\lvert Sf(z)\rvert\leq2(1-\alpha)(1-\lvert\Phi(z)\rvert^2)\left[1+\alpha\left(\dfrac{\lvert\varphi(z)\rvert^2(1-\lvert z\rvert^2)}{1-\lvert\varphi(z)\rvert^2}\right)\right].
\end{equation}
From the normalization and using the equation (\ref{preschwarz_alpha}) we conclude that $\varphi(0)=0$. Then, by Schwarz lemma, we obtain that $\lvert\varphi(z)\rvert\leq\lvert z\rvert$ for all $z\in\D$ and from this we can conclude that:
\begin{equation*}
    \dfrac{\lvert\varphi(z)\rvert^2}{1-\lvert\varphi(z)\rvert^2}\leq\dfrac{\lvert z\rvert^2}{1-\lvert z\rvert^2}.
\end{equation*}
Applying the above in (\ref{ineq_imp}) we obtain:
\begin{equation*}
    (1-\lvert z\rvert^2)^2\,\lvert Sf(z)\rvert\leq 2(1-\alpha)(1-\lvert\Phi(z)\rvert^2)(1+\alpha\lvert z\rvert^2).
\end{equation*}
We know that $1-\lvert\Phi(z)\rvert^2\leq 1$, then:
\begin{equation*}
    (1-\lvert z\rvert^2)^2\,\lvert Sf(z)\rvert\leq 2(1-\alpha)(1+\alpha)=2(1-\alpha^2).
\end{equation*}
The sharpness follows from the next example.
\end{proof}

\begin{ej}\label{ej_1}
The family of parameterized functions defined as:
\begin{equation}\label{example}
    f_\alpha(z)=\int_0^z\dfrac{1}{(1-\zeta^2)^{1-\alpha}}\,d\zeta,\ z\in\D,
\end{equation}
maximizes the Schwarzian norm defined as:
\begin{equation*}
    \lVert Sf\rVert=\sup_{z\in\D}(1-\lvert z\rvert^2)^2\,\lvert Sf(z)\rvert,
\end{equation*}
and from this, the sharpness of the inequality given in Theorem \ref{teo_2} holds for any $0\leq \alpha<1$. Notice that:

\begin{equation*}
    \dfrac{f_\alpha''}{f_\alpha'}(z)=2(1-\alpha)\dfrac{z}{1-z^2},\quad \mbox{and}\quad Sf_\alpha(z)=2(1-\alpha)\cdot\dfrac{1+\alpha z^2}{(1-z^2)^2}.
\end{equation*}
from which,
\begin{equation*}
(1-\lvert z\rvert^2)^2\,\lvert Sf_\alpha(z)\rvert=2(1-\alpha)
\dfrac{(1-\lvert z\rvert^2)^2\,\lvert 1+\alpha z^2\rvert}{\lvert 1-z^2\rvert^2} \leq 2(1-\alpha)\,\lvert1+\alpha z^2\rvert.
\end{equation*}
We can conclude that:
\begin{equation*}
    \lVert Sf_\alpha\rVert=\sup_{z\in\D}(1-\lvert z\rvert^2)^2\,\lvert Sf_\alpha(z)\rvert=2(1-\alpha)(1+\alpha)=2(1-\alpha^2).
\end{equation*}
In general, the integral formula for $f_\alpha$ given in (\ref{example}) does not give primitives in terms of elementary functions, however, when $\alpha=0$ we obtain:
\begin{equation*}
    f_0(z)=\int_0^z\dfrac{1}{1-\zeta^2}\,d\zeta=\dfrac{1}{2}\log\left(\dfrac{1+z}{1-z}\right),\end{equation*}
where $\lVert Sf_0\rVert=2$. Also, for $\alpha=1/2$ we have that:
\begin{equation*}
    f_{1/2}(z)=\int_0^z\dfrac{1}{(1-\zeta^2)^{1/2}}\,d\zeta=\sin^{-1}(z),
\end{equation*}
where $\lVert Sf_{1/2}\rVert=3/2$.
\end{ej}

\begin{rem}\label{remark_1}
In \cite{CDO} is presented the next inequality for $f\in\mathcal{C}$:

\begin{equation}\label{ineq_1}
    (1-\lvert z\rvert^2)^2\,\lvert Sf(z)\rvert+2\left\lvert\dfrac{\varphi(z)-\bar{z}}{1-z\varphi(z)}\right\rvert^2\leq 2,
\end{equation}
which is equivalent to:
\begin{equation}\label{ineq_2}
    (1-\lvert z\rvert^2)^2\,\lvert Sf(z)\rvert+2\left\rvert\bar{z}-\dfrac{1}{2}(1-\lvert z\rvert^2)\dfrac{f''}{f'}(z)\right\lvert^2\leq 2.
\end{equation}
 However, for the class $\mathcal{C}_\alpha$ we can obtain from (\ref{ineq_imp_2}) an analogous inequality associated to (\ref{ineq_1}) and (\ref{ineq_2}). In fact, we have that 
\begin{equation}\label{ineq_5}
    (1-\lvert z\rvert^2)^2\,\lvert Sf(z)\rvert\leq 2(1-\alpha)\left[(1+\alpha)-\left\lvert\dfrac{\varphi(z)-\bar{z}}{1-z\varphi(z)}\right\rvert^2-\alpha\, A\right],
\end{equation}
where,  using the equation (\ref{phi_preschwarz}), 
\begin{equation}\label{A}
    A=1-\left(\dfrac{(1-\lvert z\rvert^2)\lvert \varphi(z)\rvert}{\lvert1-z\varphi(z)\rvert}\right)^2=1-\left(\dfrac{1-\lvert z\rvert^2}{2(1-\alpha)}\left\lvert\dfrac{f''}{f'}(z)\right\rvert\right)^2.
\end{equation}
Now, from (\ref{ineq_5}) we have that:
\begin{equation*}
    (1-\lvert z\rvert^2)^2\,\lvert Sf(z)\rvert+2(1-\alpha)\left\lvert\dfrac{\varphi(z)-\bar{z}}{1-z\varphi(z)}\right\rvert^2+2(1-\alpha)\alpha\, A\leq 2(1-\alpha^2),
\end{equation*} which is equivalent to 
\begin{equation*}
    (1-\lvert z\rvert^2)^2\,\lvert Sf(z)\rvert+2(1-\alpha)\left\lvert\bar z-\dfrac{1-\lvert z\rvert^2}{2(1-\alpha)}\dfrac{f''}{f'}(z)\right\rvert^2+2(1-\alpha)\alpha\,A\leq 2(1-\alpha^2).
\end{equation*}
Thus, when $A\geq0$ we have that $\lVert Sf\rVert\leq2(1-\alpha^2)$, which occurs when $\varphi(0)=0$, because of $f''(0)=0$. In this case, using the Schwarz Lemma we can show that $A\geq 1-\lvert\varphi(z)\rvert$. Therefore we can assert the following results that improve Theorem 1 in \cite{Y}.

\end{rem}

The next theorem gives us a bound for the expression $(1-\lvert z\rvert^2)^2\,\lvert Sf(z)\rvert$ for $f\in\mathcal{C}_\alpha$ in terms of value $\lvert f''(0)\rvert$, which is not necessarily zero. To do this, we need the next lemma (its proof can be found in \cite[p. 167]{NEHARI_CONFORMAL}) which gives us a generalization of the Schwarz Lemma for bounded holomorphic functions:

\begin{lemma}\label{lemma_1}
If $\varphi:\D\rightarrow\D$ be a holomorphic function then
\begin{equation*}
    \lvert\varphi(z)\rvert\leq\dfrac{\lvert\varphi(0)\rvert+\lvert z\rvert}{1+\lvert\varphi(0)\rvert\,\lvert z\rvert}.
\end{equation*}
\end{lemma}

\begin{thm}\label{teo_3}
Let $0\leq\alpha<1$, $f\in\mathcal{C}_\alpha$ and $p=\dfrac{\lvert f''(0)\rvert}{2(1-\alpha)}$. Then:
\begin{equation*}
    (1-\lvert z\rvert^2)^2\,\lvert Sf(z)\rvert\leq 2(1-\alpha)\left(1+\alpha\,\dfrac{1+p}{1-p}\right).
\end{equation*}
\end{thm}

\begin{proof}
Let $p=\lvert\varphi(0)\rvert$. Applying Lemma \ref{lemma_1}, we have that:
\begin{equation*}
    \dfrac{\lvert\varphi(z)\rvert^2}{1-\lvert\varphi(z)\rvert^2}\leq\dfrac{(p+\lvert z\rvert)^2}{(1+p\lvert z\rvert)^2-(p+\lvert z\rvert)^2}=\dfrac{(p+\lvert z\rvert)^2}{(1-p^2)(1-\lvert z\rvert^2)}.
\end{equation*}
If we replace the above expression in (\ref{ineq_imp_2}) at the proof of Theorem \ref{teo_2} we obtain that:
\begin{equation*}
    (1-\lvert z\rvert^2)^2\,\lvert Sf(z)\rvert\leq 2(1-\alpha)(1-\lvert\Phi(z)\rvert^2)\left(1+\alpha\,\dfrac{(p+\lvert z\rvert)^2}{1-p^2}\right).
\end{equation*}
From the fact that $\lvert z\rvert<\leq 1$ and $1-\lvert\Phi(z)\rvert^2<1$ we can conclude that:
\begin{equation*}
    (1-\lvert z\rvert^2)^2\,\lvert Sf(z)\rvert\leq 2(1-\alpha)\left(1+\alpha\,\dfrac{(p+1)^2}{1-p^2}\right)=2(1-\alpha)\left(1+\alpha\,\dfrac{1+p}{1-p}\right).
\end{equation*}
\end{proof}

\begin{rem}
Although the bound obtained is not sharp in general, for the $\mathcal C_\alpha$ class, we have that in certain cases the bound in Theorem \ref{teo_3} is better than the one given in Theorem A, (see \cite{SUITA}). In fact, for $0\leq\alpha\leq 1/2$ and $f\in\mathcal{C}_\alpha$ such that $p\leq \alpha(2-\alpha)^{-1}$ then $(1+p)(1-p)^{-1}\leq (1-\alpha)^{-1}$, which implies
\begin{equation*}\label{caso_1}
    2(1-\alpha)\left(1+\alpha\,\dfrac{1+p}{1-p}\right)\leq 2.
\end{equation*}
On the other side, if $1/2<\alpha< 1$ and $f$ is such that $p\leq (3\alpha-1)(5\alpha-1)^{-1}$ then $(1+p)(1-p)^{-1}\leq (8\alpha-2)(2\alpha)^{-1}$, from which
\begin{equation*}\label{caso_2}
     2(1-\alpha)\left(1+\alpha\,\dfrac{1+p}{1-p}\right)\leq 8\alpha(1-\alpha).
\end{equation*}
\end{rem}

\section*{aknowledgment}We thank our colleague, prof. Iason Efraimidis, from Universidad Autónoma de Madrid, for his precise contribution to improve the presentation of the manuscript, and the proof of Theorem 7.

\end{document}